\documentclass[11pt,a4paper,oneside, reqno]{amsart}
\usepackage{mathrsfs}
\usepackage{amsfonts}
\usepackage{amssymb}
\usepackage{amsxtra}
\usepackage{dsfont}
\usepackage{amsthm}
\usepackage{geometry}
\usepackage{float}
\usepackage{amsmath, amssymb, amsthm}
\usepackage[utf8]{inputenc}
\usepackage[T1]{fontenc}
\usepackage{graphicx}
\usepackage[unicode, psdextra]{hyperref}
\usepackage{listings}
\usepackage{indentfirst}
\usepackage{theoremref}
\usepackage{thmtools}
\usepackage{physics}
\usepackage{mathtools}
\usepackage[polish,english]{babel}
\usepackage{color}
\usepackage{bbm}
\usepackage{url}
\usepackage{hyperref}
\usepackage[hyphenbreaks]{breakurl}

\newcommand{\pl}[1]{\foreignlanguage{polish}{#1}}

\usepackage{hyphenat}

\newcommand{\R}{\mathbb{R}}

\newcommand{\E}{\mathbb{E}}

\renewcommand{\P}{\mathbb{P}}

\newcommand{\ind}[1]{\mathbbm{1}_{#1}}

\newtheorem{theorem}{Theorem}[section]

\newtheorem{lemma}[theorem]{Lemma}

\theoremstyle{remark}
\newtheorem*{remark}{Remark}

\theoremstyle{definition}

\definecolor{green}{rgb}{0,0.75,0}

\numberwithin{equation}{section}

\title[Dim-free estimates for Riesz transforms related to Schr\"odinger operators]{Dimension-free estimates for positivity-preserving Riesz transforms related to Schr\"odinger operators with certain potentials}

\author{Maciej Kucharski}
\address{Maciej Kucharski\\
	Instytut Matematyczny\\
	Uniwersytet \pl{Wroc{\l}awski}\\
	Plac Grun\-waldzki 2\\
	50-384 \pl{Wroc{\l}aw}\\
	Poland}
\email{mkuchar@math.uni.wroc.pl}

\thanks{The author was supported by the National Science Centre (NCN), Poland, research project Preludium Bis 2019/35/O/ST1/00083.}

\begin{document}
	
\begin{abstract}
	We study the $L^\infty(\R^d)$ boundedness for Riesz transforms of the form ${V^{a}(-\frac12\Delta+V)^{-a}},$ where $a > 0$ and $V$ is a non-negative potential with power growth acting independently on each coordinate. We factorize the semigroup $e^{-tL}$ into one-dimensional factors, estimate them separately and combine the results to estimate the original semigroup. Similar results with additional assumption $a \leqslant 1$ are obtained on $L^1(\R^d)$.
\end{abstract}
	
\maketitle
\section{Introduction}
\label{sec:int}

In this paper we consider the Schr\"{o}dinger operator $L$ on $\R^d$ given by
\[
	L = -\frac{1}{2}\Delta + V,
\]
with $V$ being a non-negative potential, and the associated Riesz transform
\begin{equation} \label{eq:RV}
	R_V^a f(x) = V(x)^a L^{-a} f(x) = \frac{V(x)^a}{\Gamma(a)} \int_0^\infty e^{-tL}f(x) \, t^{a-1} \, dt, \quad a > 0.
\end{equation}

Such Riesz transforms related to Schr\"{o}dinger operators have been studied by numerous authors, see \cite{AssaadOuhabaz,AuscherBenAli,Dziubanski,DziubanskiGlowacki,Sikora,UrbanZienkiewicz}. For general $V \in L^2_{\textrm{loc}}$ Sikora proved in \cite[Theorem 1]{Sikora} that $R_V^{1/2}$ is bounded on $L^p$ for $1 < p \leqslant 2$ (in fact the result applies not only to Riesz transforms on $\R^d$ but also on more general doubling spaces), it is also well known that $R_V^1$ is bounded on $L^1$ with norm estimated by 1, see for example \cite{GallouetMorel}, \cite[Lemma 6]{KatoLp} and \cite[Theorem 4.3]{AuscherBenAli}. When the potential $V$ belongs to the reverse H\"older class $B_q$ for some $q \geqslant \frac{d}{2}$, then it is known, see \cite[Theorem 5.10]{Shen}, that $R_V^1$ is bounded on $L^1$. There are also two results regarding polynomial potentials, namely Dziubański \cite[Theorem 4.5]{Dziubanski} proved that $R_V^a$, $a > 0$, is bounded on $L^1$ and $L^\infty$ if $V$ is a polynomial and then Urban and Zienkiewicz proved in \cite[Theorem 1.1]{UrbanZienkiewicz} that $R_V^1$ is bounded on $L^\infty$ independently of the dimension for $V$ being a polynomial satisfying a certain condition of C. Fefferman. Recently it has been proved in \cite{KucharskiWrobelSchrodinger} that $R_V^a$ is bounded on $L^p$ with $0 \leqslant a \leqslant 1/p$ and $1 < p \leqslant 2$ for general $V \in L^1_{\textrm{loc}}$ and that $R_V^a$, $a > 0$, is bounded on $L^1$ and $L^\infty$ if the potential $V$ has polynomial or exponential growth.

Obtaining dimension-free bounds for the Riesz transforms related to Schr\"{o}dinger operators seems to be a significantly harder task. The only available results are the aforementioned paper by Urban and Zienkiewicz \cite{UrbanZienkiewicz}, the well-known bound for $R_V^1$ for general potentials and a result regarding a particular case of $R_V^{1/2}$ with $V(x) = \abs{x}^2$, see \cite{HarboureRosaSegoviaTorrea,LustPiquard,Kucharski}. Our goal is to extend these dimension-free results and get $L^\infty$ bounds for $R_V^a$ with $a > 0$ and $L^1$ bounds for $R_V^a$ with $a \leqslant 1$ when the potential $V$ is of the form
\begin{equation} \label{eq:V}
	V(x) = V_1(x) + \dots + V_d(x),
\end{equation}
where each $V_i$ acts only on the $i$-th coordinate of the argument $x$ and has polynomial growth with the exponent not greater than 2, i.e. there are absolute constants $m$ and $M$ such that
\begin{equation} \label{eq:VmM}
	m \abs{x_i}^\alpha \leqslant V_i(x) \leqslant M \abs{x_i}^\alpha
\end{equation}
for some $0 < \alpha \leqslant 2$. This holds for example if $V_i(x) = x_i^2$ and $V(x) = \abs{x}^2$, which results in the operator $L = -\frac{1}{2} \Delta + \abs{x}^2$ called the harmonic oscillator. The reason why we can only handle $\alpha \leqslant 2$ is related to the distribution of the Brownian motion, which arises in the Feynman--Kac formula \eqref{eq:feynman-kac3}, and is visible in \eqref{eq:alpha2}.

By the definition \eqref{eq:RV} of $R_V^a$ and the positivity-preserving property of the semigroup $e^{-tL}$ obtaining the $L^\infty$ bounds for $R_V^a$ amounts to estimating the value of $R_V^a(\ind{})(x)$ independently of $x$ and $d$, which in turn hints that the main part of the proof is estimating the semigroup applied to the constant function 1, i.e. $e^{-tL}(\ind{})$. The particular structure of $V$ \eqref{eq:V} lets us write
\begin{equation} \label{eq:Li}
	L = \sum_{i=1}^d L_i, \quad \text{where } L_i = -\frac{1}{2}\frac{\partial^2}{\partial x_i^2} + V_i,
\end{equation}
and, as a consequence, factorize the semigroup $e^{-tL}$ in the following way
\begin{equation} \label{eq:fact0}
	e^{-tL} = \prod_{i=1}^d e^{-tL_i} \quad \text{and hence} \quad e^{-tL}(\ind{})(x) = \prod_{i=1}^d e^{-tL_i}(\ind{})(x).
\end{equation}
This is the key property allowing us to get estimates that does not depend on the dimension $d$.

The main result of the paper is the following theorem.
\begin{theorem} \thlabel{thm1}
	Fix $0 < \alpha \leqslant 2$ and let $V$ given by \eqref{eq:V} satisfy \eqref{eq:VmM}. For $a > 0$ let the Riesz transform $R_V^a$ be defined as in \eqref{eq:RV}. Then there is a constant $C > 0$ depending on $m$, $M$, and $\alpha$ and independent of the dimension $d$ such that
	\[
		\norm{R_V^a f}_{L^\infty(\R^d)} \leqslant C\norm{f}_{L^\infty(\R^d)}, \quad f \in L^\infty(\R^d).
	\]
\end{theorem}
As a by-product of our considerations we also obtain $L^1$ estimates for $R_V^a$, but only for a limited range of $a$. The reason for this is that we need to use concavity of the function $x^a$.
\begin{theorem} \thlabel{thm2}
	Fix $0 < \alpha \leqslant 2$ and let $V$ given by \eqref{eq:V} satisfy \eqref{eq:VmM}. For $0 < a \leqslant 1$ let the Riesz transform $R_V^a$ be defined as in \eqref{eq:RV}. Then there is a constant $C > 0$ depending on $m$, $M$, and $\alpha$ and independent of the dimension $d$ such that
	\[
		\norm{R_V^a f}_{L^1(\R^d)} \leqslant C\norm{f}_{L^1(\R^d)}, \quad f \in L^1(\R^d).
	\]
\end{theorem}

\begin{remark}
	For technical reasons we will assume that $d \geqslant 3$. The case of $d = 1, 2$ follows from previous results, e.g. \cite{KucharskiWrobelSchrodinger}.
\end{remark}

\subsection{Structure and methods}
The main part of the proof is contained in Section \ref{sec:semigroup} where we prove that the one-dimensional semigroups $e^{-tL_i}$ decay exponentially in $t$ and $V(x)$ for small values of $t$, i.e. we have
\[
	e^{-tL_i}(\ind{})(x) \leqslant e^{-c_N t V_i(x)} \quad \textrm{for } t \leqslant N.
\]
It is noteworthy that the constant in front of the exponential in the above estimate is 1, which means that we can multiply one-dimensional bounds to estimate the full semigroup $e^{-tL}$ without constants growing with the dimension. The proof is divided into three cases depending on the value of $\abs{x_i}$ and $tV_i(x)$ but in all of them the main ingredient is the Feynman--Kac formula \eqref{eq:feynman-kac3}.

In Section \ref{sec:Linf} we use results from Section \ref{sec:semigroup} and a similar result \cite[Lemma 4.1]{KucharskiWrobelSchrodinger} giving an exponential decay of the semigroup for large values of $t$, namely
\[
	e^{-tL_i}(\ind{})(x) \leqslant e^{-ct} \quad \textrm{for } t \geqslant N,
\]
to estimate the $L^\infty$ norm of $R_V^a$.

Finally in Section \ref{sec:L1} we estimate the $L^1$ norm of the Riesz transform. We use duality between $L^\infty$ and $L^1$ which reduces estimating the $L^1$ norm of the operator $R_V^a = V^a L^{-a}$ to estimating the $L^\infty$ norm of the adjoint operator
\[
	(L^{-a} V^a) f(x) = \frac{1}{\Gamma(a)} \int_0^\infty e^{-tL}(V^a f)(x) \, t^{a-1} \, dt.
\]
Again, using the positivity-preserving property of the semigroup $e^{-tL}$ reduces the task to estimating $e^{-tL}(V^a)$. In this case, although the factorization \eqref{eq:fact0} of the semigroup as an operator still applies, it does not behave well when the semigroup is applied to $V^a$ instead of the constant function, so we use the following formula
\[
	e^{-tL}(V) = \sum_{i=1}^d e^{-tL}(V_i) = \sum_{i=1}^d e^{-tL^i}(\ind{}) \, e^{-tL_i}(V_i), \quad \text{where } L^i = L - L_i.
\]

\subsection{Notation}
We conclude the introduction by establishing some useful notation used throughout the paper.
\begin{enumerate}
	\item We abbreviate $L^p(\R^d)$ to $L^p$ and $\norm{\cdot}_{L^p}$ to $\norm{\cdot}_p$. For a linear operator $T$ acting on $L^p$ we denote its norm by $\norm{T}_{p \to p}$.
	\item By $\ind{}$ we denote the constant function 1 and by $\ind{X}$ we denote the characteristic function of the set $X$.
	\item The space of smooth compactly supported functions on $\R^d$ is denoted by $C_c^\infty$.
	\item For two quantities $A$ and $B$ we write $A \lesssim B$ if $A \leqslant CB$ for some constant $C > 0$ which may depend on $m$, $M$ and $\alpha$ and is independent of the dimension $d$. If $A \lesssim B$ and $B \lesssim A$, then we write $A \approx B$.
	\item For $x \in \R^d$ we denote its components by $x_1, \dots, x_d$, i.e. $x = (x_1, \dots, x_d)$.
	\item For a random variable $X$ defined on a probability space $(\Omega, \mathcal{F}, \P)$ and $A \subseteq \R$ we denote $\P(X \in A) \coloneqq \P\left( \left\{ \omega \in \Omega: X(\omega) \in A \right\} \right)$.
\end{enumerate}

\section{Definitions}
\label{sec:def}

We begin by defining the semigroup $e^{-tL}$ and then we proceed to defining the Riesz transform $R_V^a$. By the result of Kato \cite[p. 137]{Kato} the operator $L = -\frac{1}{2}\Delta + V$ is essentially self-adjoint on $C_c^\infty$ and hence it has a non-negative self-adjoint extension. This in turn means that $L$ generates a strongly continuous semigroup of contractions on $L^2$ which can be expressed using the Feynman--Kac formula
\begin{equation} \label{eq:feynman-kac}
	e^{-tL}f(x) = \E_x \left[ e^{-\int_0^t V(X_s) \, ds} f(X_t) \right], \quad f \in L^2,
\end{equation}
where the expectation $\E_x$ is taken with respect to the Wiener measure of the standard $d$-dimensional Brownian motion $\{X_s\}_{s>0}$ starting at $x\in \R^d;$ here $X_s=(X_s^1,\ldots,X_s^d).$ Since the right-hand makes sense also for $f \in L^\infty$, we use the Feynman--Kac formula to define $e^{-tL}$ acting on $L^\infty$ as
\begin{equation} \label{eq:feynman-kac2}
	e^{-tL}f(x) \coloneqq \E_x \left[ e^{-\int_0^t V(X_s) \, ds} f(X_t) \right], \quad f \in L^\infty.
\end{equation}
Similarly, using the fact that $V$, and hence $L$, act on each coordinate separately, see \eqref{eq:V} and \eqref{eq:Li}, we define one-dimensional semigroups $e^{-tL_i}$, $i = 1, \dots, d$, as follows
\begin{equation} \label{eq:feynman-kac3}
	e^{-tL_i}f(x) \coloneqq \E_{x_i} \left[ e^{-\int_0^t V_i(X_s) \, ds} f_{x^i}(X_t^i) \right], \quad f \in L^\infty,
\end{equation}
where
\[
	f_{x^i}(y) = f(x_1, \dots, x_{i-1}, y, x_{i+1}, \dots, x_d).
\]
Here the expectation $\E_{x_i}$ is taken with regards to the Wiener measure of the standard one-dimensional Brownian motion $\{X_s^i\}_{s > 0}$ starting at $x_i \in \R$.

As the next lemma shows, this is the definition that suits best our purpose of factorizing the semigroup $e^{-tL}$ into one-dimensional factors $e^{-tL_i}$.

\begin{lemma} \thlabel{lem:schrodinger_dimfree_fact}
	Fix $d$ and let the $d$-dimensional semigroup $e^{-tL}$ be given by \eqref{eq:feynman-kac2} and the one-dimensional semigroup $e^{-tL_i}$ by \eqref{eq:feynman-kac3}. Then for $f \in L^\infty$ we have
	\begin{equation} \label{eq:fact1}
		e^{-tL} f(x) = \left( \left(\prod_{i=1}^d e^{-tL_i} \right) f \right)(x) \quad \text{and} \quad e^{-tL}(\ind{})(x) = \prod_{i=1}^d \left( e^{-tL_i}(\ind{})(x) \right).
	\end{equation}
\end{lemma}

\begin{proof}
	We will prove by induction that for $k = 1, \dots, d$ we have
	\begin{equation} \label{eq:fact_ind}
		\left( \left(\prod_{i=1}^k e^{-tL_i} \right) f \right) (x) = \E_{(x_1, \dots, x_k)} \left[ e^{-\int_0^t \sum_{i=1}^k V_i(X_s) \, ds} f(X_t^1, \dots, X_t^k, x_{k+1}, \dots, x_d) \right],
	\end{equation}
	which justifies the first formula in \eqref{eq:fact1} if we take $k = d$.
	
	The case $k=1$ is clear from the definition \eqref{eq:feynman-kac3} of $e^{-tL_1}$. Now suppose that \eqref{eq:fact_ind} holds. Then
	\begin{align*}
		&\left( \left(\prod_{i=1}^{k+1} e^{-tL_i} \right) f \right) (x) = \E_{x_{k+1}} \left[ e^{-\int_0^t V_{k+1}(X_s) \, ds} \left( \left(\prod_{i=1}^k e^{-tL_i} \right) f \right)_{x^{k+1}}(X_t^{k+1}) \right] \\
		&= \E_{x_{k+1}} \left[ e^{-\int_0^t V_{k+1}(X_s) \, ds} \, \E_{(x_1, \dots, x_k)} \left[ e^{-\int_0^t \sum_{i=1}^k V_i(X_s) \, ds} f(X_t^1, \dots, X_t^k, X_t^{k+1}, x_{k+2}, \dots, x_d) \right] \right] \\
		&= \E_{(x_1, \dots, x_{k+1})} \left[ e^{-\int_0^t \sum_{i=1}^{k+1} V_i(X_s) \, ds} f(X_t^1, \dots, X_t^{k+1}, x_{k+2}, \dots, x_d) \right].
	\end{align*}
	Note that we can use the same Brownian motion in the inner and in the outer expected value since its coordinates are independent of each other and $V_i(X_s)$ depends only on $X_s^i$.
	
	The second formula in \eqref{eq:fact1} follows from the definitions of $e^{-tL}$ and $e^{-tL_i}$ and the fact that the coordinates of $d$-dimensional Brownian motion are independent.
\end{proof}

Now we take $a > 0$ and a non-negative function $f \in L^\infty$ and define the Riesz transform
\begin{equation} \label{eq:RV2}
	R_V^a f(x) = \frac{V(x)^a}{\Gamma(a)} \int_0^\infty e^{-tL}f(x) \, t^{a-1} \, dt,
\end{equation}
where $e^{-tL}f(x)$ is defined as in \eqref{eq:feynman-kac2}. Lastly, we use the positivity-preserving property of the semigroup $e^{-tL}$, which means that $e^{-tL} f \geqslant 0$ whenever $f \geqslant 0$, to rewrite the main theorem in a simpler form. Namely, we have
\[
	\abs{e^{-tL}f(x)} \leqslant e^{-tL}\left( \norm{f}_\infty \ind{} \right)(x) = \norm{f}_\infty \, e^{-tL}(\ind{})(x), \quad f \in L^\infty,
\]
which means that the Riesz transform $R_V^a$ is bounded on $L^\infty$ if
\[
	\norm{R_V^a(\ind{})} < \infty
\]
with its norm being
\[
	\norm{R_V^a}_{\infty \to \infty} = \norm{R_V^a(\ind{})}_{\infty}.
\]
Thus, \thref{thm1} can be rewritten as
\begin{theorem} \thlabel{thm1'}
	Fix $0 < \alpha \leqslant 2$ and let $V$ given by \eqref{eq:V} satisfy \eqref{eq:VmM}. For $a > 0$ let the Riesz transform $R_V^a$ be defined as in \eqref{eq:RV2}. Then there is a constant $C > 0$ depending on $m$, $M$, and $\alpha$ and independent of the dimension $d$ such that
	\[
		\norm{R_V^a (\ind{})}_{\infty} \leqslant C.
	\]
\end{theorem}

\section{One-dimensional estimates}
\label{sec:semigroup}
In this section we prove the aforementioned exponential decay of the one-dimensional semigroup which we will then combine to estimate the semigroup $e^{-tL}$.
\begin{lemma} \thlabel{lemma}
	For every $N > 0$ there is a constant $c_N > 0$ such that
	\begin{equation} \label{eq:lemma1}
		e^{-tL_i}(\ind{})(x) \leqslant e^{-c_N tV_i(x)}
	\end{equation}
	for all $x \in\R^d$ and $0 \leqslant t \leqslant N$. Moreover, if $\abs{x_i} \leqslant 4$, then
	\begin{equation} \label{eq:lemma2}
		e^{-tL_i}(\ind{})(x) \leqslant e^{-c_N\left( t^{\frac{\alpha}{2}+1} + tV_i(x) \right)}, \quad t \leqslant N.
	\end{equation}
\end{lemma}

\begin{proof}
	First we will show that \eqref{eq:lemma1} is satisfied for $0 \leqslant t \leqslant t_0$ for some $t_0$ and then we will extend the estimate to all $0 \leqslant t \leqslant N$.
	
	We begin with the case $\abs{x_i} \leqslant 4$. We will make use of the inequality
	\begin{equation} \label{eq:exptaylor}
		e^{-x} \leqslant 1 - x + \frac{x^2}{2}, \qquad x \geqslant 0.
	\end{equation}
	The Feynman--Kac formula \eqref{eq:feynman-kac3} together with \eqref{eq:exptaylor} give
	\begin{equation} \label{eq:fktaylor}
		e^{-tL_i}(\ind{})(x) \leqslant 1 - \E_{x_i} \left[ \int_0^t V_i(X_s) \, ds \right] + \frac{1}{2} \E_{x_i} \left[ \left( \int_0^t V_i(X_s) \, ds \right)^2 \right].
	\end{equation}
	We need to estimate the first and the second expected value in the expression above. In order to do this we will use the fact that for any $a,b \geqslant 0$ and $\alpha > 0$ we have
	\begin{equation} \label{eq:abalpha}
		(a + b)^\alpha \approx a^\alpha + b^\alpha,
	\end{equation}
	and an estimate for the moments of normal distribution
	\begin{equation}
		\E \abs{X_s^i}^\alpha \approx s^{\alpha/2}.
	\end{equation}
	
	Let us begin by estimating $\E_{x_i} V_i(X_s)$ from below and assume without loss of generality that $x_i \geqslant 0$.
	\begin{align*}
		\E_{x_i} \, V_i(X_s) \gtrsim \E_0 \abs{X_s^i + x_i}^\alpha \geqslant \E_0 \left[ \ind{\{X_s^i \geqslant 0\}} (X_s^i + x_i)^\alpha \right] \approx s^{\alpha/2} + x_i^{\alpha}
	\end{align*}
	Integrating this gives
	\[
		\E_{x_i} \left[ \int_0^t V_i(X_s) \, ds \right] \gtrsim t^{\frac{\alpha}{2}+1} + tx_i^\alpha.
	\]
	Now we estimate the last term in \eqref{eq:fktaylor} using Cauchy--Schwarz inequality.
	\begin{align*}
		\E_{x_i} \left[ \left( \int_0^t V_i(X_s) \, ds \right)^2 \right] &\lesssim t \int_0^t \E_0 \left[ \abs{X_s^i + x_i}^{2\alpha} \right] ds \lesssim t \int_0^t \E_0 \left[ \abs{X_s^i}^{2\alpha} + x_i^{2\alpha} \right] ds \\
		&\approx t \int_0^t s^\alpha + x_i^{2\alpha} ds \approx t^{\alpha+2} + t^2 x_i^{2\alpha} \lesssim \left( t^{\frac{\alpha}{2}+1} + tx_i^\alpha \right)^2.
	\end{align*}
	Plugging this into \eqref{eq:fktaylor}, recalling that $\abs{x_i} \leqslant 4$, and choosing $t_0$ sufficiently small yields
	\begin{align*}
		e^{-tL_i}(\ind{})(x) &\leqslant 1 - c_1 \left( t^{\frac{\alpha}{2}+1} + tx_i^\alpha \right) + c_2 \left( t^{\frac{\alpha}{2}+1} + tx_i^\alpha \right)^2 \\
		&\leqslant 1 - c \left( t^{\frac{\alpha}{2}+1} + tx_i^\alpha \right) \leqslant e^{-c \left( t^{\frac{\alpha}{2}+1} + tx_i^\alpha \right)}
	\end{align*}
	which implies \eqref{eq:lemma1} and \eqref{eq:lemma2} for $t \leqslant t_0$.
	
	The second case is when $\abs{x_i} > 4$ and $tV_i(x) \leqslant 2A\log 5$, where $A = \frac{2^\alpha M}{m}$ with $m$ and $M$ as in \eqref{eq:VmM}. We will roughly show that then we have
	\begin{equation} \label{eq:dtetL}
		\frac{d}{dt} e^{-tL_i}(\ind{})(x) = -e^{-tL_i}(V_i)(x) \leqslant -cV_i(x).
	\end{equation}
	However since the equality may not hold, we replace $V_i$ with $V_i^n(x) = \min(V_i(x), n)$ for any $n > 0$, establish \eqref{eq:dtetL} for $V_i^n$, then we prove \eqref{eq:lemma1} for $V_i^n$ and finally we deduce \eqref{eq:lemma1} for $V_i$.
	
	Recall that $V_i$ satisfies $m\abs{x_i}^\alpha \leqslant V_i(x) \leqslant M\abs{x_i}^\alpha$ and take $x_i, y_i \in \R$ such that $\abs{x_i-y_i} \leqslant \frac{\abs{x_i}}{2}$. Then $\frac{\abs{x_i}}{2} \leqslant \abs{y_i} \leqslant 2\abs{x_i}$ so that we have
	\[
		V_i(y) \leqslant M\abs{y_i}^\alpha \leqslant 2^\alpha M \abs{x_i}^\alpha \leqslant \frac{2^\alpha M}{m} V_i(x) = AV_i(x)
	\]
	and
	\[
		V_i(y) \geqslant m\abs{y_i}^\alpha \geqslant m \frac{\abs{x_i}^\alpha}{2^\alpha} \geqslant \frac{m}{2^\alpha M} V_i(x) = \frac{1}{A} V_i(x)
	\]
	We also calculate the probability that $\sup_{0 \leqslant s \leqslant t} \abs{X_s^i-x_i} \geqslant \frac{\abs{x_i}}{2}$ using the reflection principle to get
	\begin{equation} \label{eq:prob}
		\P \left( \sup_{0 \leqslant s \leqslant t} \abs{X_s^i-x_i} \geqslant \frac{\abs{x_i}}{2} \right) \leqslant 4e^{-\frac{\abs{x_i}^2}{8t}}.
	\end{equation}
	
	Now, for $n > 0$, we define $V_i^n(x) = \min(V_i(x), n)$ and $L_i^n = -\frac{1}{2}\frac{\partial^2}{\partial x_i^2} + V_i^n$ and use the Feynman--Kac formula and \eqref{eq:prob} to get
	\begin{align*}
		e^{-tL_i^n}(V_i^n)(x) &= \E_{x_i} \left[ e^{-\int_0^t V_i^n(X_s) \, ds} V_i^n(X_t) \right] \geqslant \E_{x_i} \left[ e^{-\int_0^t V_i(X_s) \, ds} V_i^n(X_t) \right] \\
		&\geqslant \P\left( \forall_{0 \leqslant s \leqslant t} \ \tfrac{V_i^n(x)}{A} \leqslant V_i^n(X_s) \text{ and } V_i(X_s) \leqslant AV_i(x) \right) \frac{V_i^n(x)}{A} e^{-AtV_i(x)} \\
		&\geqslant \P\left( \forall_{0 \leqslant s \leqslant t} \ \tfrac{V_i(x)}{A} \leqslant V_i(X_s) \leqslant AV_i(x) \right) \frac{V_i^n(x)}{A} e^{-AtV_i(x)} \\
		&\geqslant \frac{V_i^n(x)}{A} \left( 1 - 8e^{-\frac{\abs{x_i}^2}{8t}} \right) e^{-2A^2 \log 5} \\
		&\geqslant \frac{V_i^n(x)}{A} \left( 1 - 8e^{-\frac{4^2}{8t_0}} \right) e^{-2A^2 \log 5} \geqslant cV_i^n(x)
	\end{align*}
	if $t_0$ is sufficiently small, which proves that
	\begin{equation} \label{eq:dtetL'}
		\frac{d}{dt} e^{-tL_i^n}(\ind{})(x) = -e^{-tL_i^n}(V_i^n)(x) \leqslant -cV_i^n(x).
	\end{equation}
	Differentiation is allowed here by the Leibniz integral rule. Now we show that this implies a version of \eqref{eq:lemma1} with $V_i^n$. Consider the function
	\[
		f(t) = e^{-tL_i^n}(\ind{})(x) \, e^{ctV_i^n(x)}.
	\]
	If we differentiate it and use \eqref{eq:dtetL'}, we get
	\begin{align*}
		f'(t) &= \frac{d}{dt} e^{-tL_i^n}(\ind{})(x) \, e^{ctV_i^n(x)} + cV_i^n(x) \, e^{-tL_i^n}(\ind{})(x) \, e^{ctV_i^n(x)} \\
		&\leqslant -cV_i^n(x) \, e^{ctV_i^n(x)} + cV_i^n(x) \, e^{ctV_i^n(x)} = 0.
	\end{align*}
	Since $f(0) = 1$, we conclude that
	\[
		e^{-tL_i^n}(\ind{})(x) \leqslant e^{-ctV_i^n(x)}.
	\]
	Now we take the limit as $n$ goes to infinity on both sides of the inequality. The left-hand side becomes
	\[
		\lim_{n \to \infty} e^{-tL_i^n}(\ind{})(x) = \lim_{n \to \infty} \E_{x_i} \left[ e^{-\int_0^t V_i^n(X_s) \, ds} \right] = \E_{x_i} \left[ e^{-\int_0^t V_i(X_s) \, ds} \right] = e^{-tL_i}(\ind{})(x).
	\]
	Passing with the limit under the integral sign is allowed since the integrand is dominated by the constant function $\ind{}$ which is integrable.
	On the right-hand side we get
	\[
		\lim_{n \to \infty} e^{-ctV_i^n(x)} = e^{-ctV_i(x)},
	\]
	so altogether we get \eqref{eq:lemma1}.
	
	The last case to consider is $\abs{x_i} > 4$ and $tV_i(x) > 2A\log 5$. We choose sufficiently small $t_0$ and use \eqref{eq:prob} to obtain
	\begin{equation} \label{eq:alpha2}
		\begin{aligned}
			e^{-tL_i}(\ind{})(x) &\leqslant e^{-\frac{tV_i(x)}{A}} \P \left( \forall_{0 \leqslant s \leqslant t} \ \tfrac{V_i(x)}{A} \leqslant V_i(X_s) \right) + 1 \cdot \P \left( \exists_{0 \leqslant s \leqslant t} \ \tfrac{V_i(x)}{A} > V_i(X_s) \right) \\
			&\leqslant e^{-\frac{tV_i(x)}{A}} + 4e^{-\frac{\abs{x_i}^2}{8t}} \leqslant 5e^{-\frac{tV_i(x)}{A}} \leqslant e^{-\frac{tV_i(x)}{2A}},
		\end{aligned}
	\end{equation}
	which is \eqref{eq:lemma1}. In the second-to-last inequality we used the assumption $\alpha \leqslant 2$.
	
	Recall that we have just proved that
	\[
		e^{-tL_i}(\ind{})(x) \leqslant e^{-c tV_i(x)}
	\]
	is satisfied for $t \leqslant t_0$ and $x \in \R^d$. If $N \leqslant t_0$, then the proof is finished, so suppose that $N > t_0$ and take $t \in [t_0, N]$. Then we have
	\[
		e^{-tL_i}(\ind{})(x) \leqslant e^{-t_0L_i}(\ind{})(x) \leqslant e^{-ct_0V_i(x)} = e^{-c\frac{t_0}{t} tV_i(x)} \leqslant e^{-c\frac{t_0}{N} tV_i(x)} = e^{-c_N tV_i(x)}.
	\]
	The inequality \eqref{eq:lemma2} can be extended to $t \in [0, N]$ in a very similar way. Suppose that $N > t$ and take $t \in [t_0, N]$. Then
	\[
		e^{-tL_i}(\ind{})(x) \leqslant e^{-c\left( t_0^{\frac{\alpha}{2}+1} + t_0 V_i(x) \right)} \leqslant e^{-c\left( \left( \frac{t_0}{N} \right)^{\frac{\alpha}{2}+1} t^{\frac{\alpha}{2}+1} + \frac{t_0}{N} t V_i(x) \right)} = e^{-c_N\left( t^{\frac{\alpha}{2}+1} + tV_i(x) \right)}.
	\]
	This finishes the proof.
\end{proof}

\section{$L^\infty$ dimension-free estimates --- proof of \thref{thm1'}}
\label{sec:Linf}
In this section we prove \thref{thm1'} using one-dimensional estimates from \thref{lemma}. The other relevant result is \cite[Lemma 4.1]{KucharskiWrobelSchrodinger}, which guarantees that there exist universal constants $C > 0$ and $\delta' > 0$ such that
\[
	e^{-tL_i}(\ind{})(x) \leqslant Ce^{-\delta' t}
\]
for all $i = 1, \dots, d$ and $x \in \R^d$. This in turn means, thanks to \eqref{eq:fact1}, that we have
\begin{equation} \label{eq:ectd}
	e^{-tL}(\ind{})(x) \leqslant e^{-d\delta t}
\end{equation}
for $x \in \R^d$ and $t \geqslant N$, where $N > 0$ and $\delta > 0$ are universal constants.

First we estimate the upper part of the integral in \eqref{eq:RV2}, i.e. the integral from $N$ to $\infty$, dividing the calculations into two cases depending on the value of $a$. If $a < 1$, then
\begin{align*}
	V(x)^a\int_N^\infty &e^{-tL}(\ind{})(x) \, t^{a-1} \, dt \leqslant V(x)^a \, e^{-\frac{N}{2}L}(\ind{})(x) \int_N^\infty e^{-\frac{t}{2}\delta d} \, t^{a-1} \, dt \\
	&\lesssim \frac{N^{a-1}}{\delta d} V(x)^a \, e^{-\frac{N}{2}L}(\ind{})(x) \lesssim \frac{1}{d} \sum_{i=1}^d V_i(x)^a \, e^{-\frac{N}{2}L_i}(\ind{})(x).
\end{align*}
In the last inequality we used the fact that
\[
	\left( x_1 + \dots + x_d \right)^a \leqslant x_1^a + \dots x_d^a
\]
for $a \leqslant 1$ and $x_i \geqslant 0$.

If, on the other hand, $a \geqslant 1$, then
\begin{align*}
	V(x)^a\int_N^\infty &e^{-tL}(\ind{})(x) \, t^{a-1} \, dt \leqslant V(x)^a \, e^{-\frac{N}{2}L}(\ind{})(x) \int_N^\infty e^{-\frac{t}{2}\delta d} \, t^{a-1} \, dt \\
	&\lesssim \frac{1}{(\delta d)^a} V(x)^a \, e^{-\frac{N}{2}L}(\ind{})(x) \lesssim \frac{1}{d} \sum_{i=1}^d V_i(x)^a \, e^{-\frac{N}{2}L_i}(\ind{})(x).
\end{align*}
Here in the last inequality we used that fact that
\[
	\left( x_1 + \dots + x_d \right)^a \leqslant d^{a-1} \left( x_1^a + \dots + x_d^a \right)
\]
for $a \geqslant 1$ and $x_i \geqslant 0$, which follows from Jensen's inequality or H\"{o}lder's inequality.
Thus, we have reduced our problem to the one-dimensional case of estimating $V_i^a e^{-\frac{N}{2}L_i}(\ind{})$ which can be done by invoking \eqref{eq:lemma1}, namely
\begin{equation} \label{eq:xexmax}
	V_i(x)^a \, e^{-\frac{N}{2}L_i}(\ind{})(x) \leqslant V_i(x)^a \, e^{-\frac{N}{2} c_N V_i(x)} \leqslant \left( \frac{2a}{N c_N e} \right)^a
\end{equation}

Then we handle the lower part of the integral in \eqref{eq:RV2}. We estimate $e^{-tL}(\ind{})(x)$ for $t \leqslant N$ independently of $x$ and $d$ by using \eqref{eq:lemma1} and the factorization property \eqref{eq:fact1}, which gives
\[
	e^{-tL}(\ind{})(x) \leqslant e^{-c_NtV(x)},
\]
and then integrate
\[
	V(x)^a \int_0^{N} e^{-tL}(\ind{})(x) \, t^{a-1} \, dt \leqslant V(x)^a \int_0^\infty e^{-c_NtV(x)} \, t^{a-1} \, dt \lesssim c_N^{-a}.
\]
This completes the proof of \thref{thm1'}.

\section{$L^1$ dimension-free estimates}
\label{sec:L1}
In this section we will again use the one-dimensional estimates for the semigroups $e^{-tL_i}$ to prove dimension-free estimates of the $L^1$ norm of $R_V^a$ for $0 < a \leqslant 1$. The idea is to estimate the $L^\infty$ norm of the adjoint operator formally given by
\[
	(L^{-a} V^a) f(x) = \frac{1}{\Gamma(a)} \int_0^\infty e^{-tL}(V^a f)(x) \, t^{a-1} \, dt.
\]
As before, the positivity-preserving property of $e^{-tL}$ lets us reduce the task to estimating the $L^\infty$ norm of
\begin{equation} \label{eq:L1def}
	L^{-a}(V^a)(x) = \frac{1}{\Gamma(a)} \int_0^\infty e^{-tL}(V^a)(x) \, t^{a-1} \, dt.
\end{equation}
However, since $V^a$ may be unbounded, it is not clear if the integral above is a measurable function of $x$. The issue was addressed in detail in \cite[Section 5]{KucharskiWrobelSchrodinger}. Briefly, we define
\begin{equation} \label{eq:L1def2}
	L^{-a}(V^a)(x) \coloneqq \lim_{N \to \infty } \frac{1}{\Gamma(a)} \int_0^\infty e^{-tL}(V^a\ind{\abs{V}<N})(x) \, t^{a-1} \, e^{-t/N} \, dt
\end{equation}
and we see that each integral is finite and measurable by \cite[Lemma 3.1]{KucharskiWrobelSchrodinger}, hence the limit is also measurable. Later it will turn out that the integral in \eqref{eq:L1def} is finite, which lets us handle \eqref{eq:L1def} instead of using \eqref{eq:L1def2}. As in the $L^\infty$ case this lets us reformulate \thref{thm2} in the following way
\begin{theorem} \thlabel{thm2'}
	Fix $0 < \alpha \leqslant 2$ and let $V$ given by \eqref{eq:V} satisfy \eqref{eq:VmM}. For $0 < a \leqslant 1$ let the Riesz transform $R_V^a$ be defined as in \eqref{eq:RV2}. Then there is a constant $C > 0$ depending on $m$, $M$, and $\alpha$ and independent of the dimension $d$ such that
	\[
		\norm{L^{-a} (V^a)}_{\infty} \leqslant C.
	\]
\end{theorem}
Before we move to the proof, we need two general results regarding the semigroup $e^{-tL}$. The first one is a factorization property for $e^{-tL}(V)$
\begin{equation} \label{eq:factL1}
	e^{-tL}(V) = \sum_{i=1}^d e^{-tL}(V_i) = \sum_{i=1}^d e^{-tL^i}(\ind{}) \, e^{-tL_i}(V_i), \quad \text{where } L^i = L - L_i.
\end{equation}
The second one is an estimate for $e^{-tL_i}(V_i^a)$
\begin{align} \label{eq:etLV}
	e^{-tL_i}(V_i^a)(x) &= \E_{x_i} \left[ e^{-\int_0^t V_i(X_s) \, ds} \, V_i(X_t)^a \right] \nonumber \\
	&\lesssim \E_0 \left[ e^{-\int_0^t V_i(X_s + x) \, ds} \, V_i(X_t)^a \right] + \E_{0} \left[ e^{-\int_0^t V_i(X_s + x) \, ds} \, V_i(x)^a \right] \nonumber \\
	&\lesssim \E_0 \left[ V_i(X_t)^a \right] + V_i(x)^a \, \E_{x_i} \left[ e^{-\int_0^t V_i(X_s) \, ds} \right] \nonumber \\
	&\lesssim t^{\frac{a\alpha}{2}} + V_i(x)^a \, e^{-tL_i}(\ind{})(x),
\end{align}
valid for $t > 0$ and $x \in \R^d$. Here we used estimate \eqref{eq:VmM} for $V$ and \eqref{eq:abalpha}.
Now we are in position to prove \thref{thm2'}
\begin{proof}[Proof of \thref{thm2'}]
	We begin with the upper part of the integral in \eqref{eq:L1def}, i.e. the integral from $N$ to $\infty$. Using subadditivity of the function $x^a$ for $a \leqslant 1$, factorization \eqref{eq:factL1} and \eqref{eq:ectd} we obtain
	\begin{align*}
		\int_N^\infty e^{-tL}(V^a)(x) \, t^{a-1} \, dt &\leqslant \int_N^\infty \sum_{i=1}^d e^{-tL}(V_i^a)(x) \, t^{a-1} \, dt \\
		&\leqslant \int_N^\infty \sum_{i=1}^d e^{-tL^i}(\ind{})(x) \, e^{-tL_i}(V_i^a)(x) \, t^{a-1} \, dt \\
		&\leqslant \int_N^\infty \sum_{i=1}^d e^{-t \delta (d-1)} e^{-tL_i}(V_i^a)(x) \, t^{a-1} \, dt \\
		&\leqslant N^{a-1} \sum_{i=1}^d \int_N^\infty \norm{e^{-tL_i}(V_i^a)}_\infty e^{-t \delta (d-1)} \, dt \\
		&\lesssim \sum_{i=1}^d \norm{e^{-NL_i}(V_i^a)}_\infty \int_N^\infty e^{-t \delta (d-1)} \, dt \\
		&\lesssim \frac{1}{d-1} \sum_{i=1}^d \norm{e^{-NL_i}(V_i^a)}_\infty.
	\end{align*}
	Then we use \eqref{eq:etLV} and \eqref{eq:lemma1} and we estimate the resulting function similarly to \eqref{eq:xexmax}.
	
	To deal with the lower part we use the inequality
	\[
		e^{-tL}(V^a) \leqslant e^{-tL}(V)^a, \quad a \leqslant 1,
	\]
	which follows from H\"{o}lder's inequality. We use this and \eqref{eq:factL1} to get
	\[
		\int_0^N e^{-tL}(V^a)(x) \, t^{a-1} \, dt \leqslant \int_0^N \left( \sum_{i=1}^d e^{-tL^i}(\ind{})(x) \, e^{-tL_i}(V_i)(x) \right)^a \, t^{a-1} \, dt.
	\]
	Then we use \eqref{eq:etLV} and obtain
	\begin{align*}
		\int_0^N e^{-tL}(V^a)(x)& \, t^{a-1} \, dt \lesssim \int_0^N \left( \sum_{i=1}^d e^{-tL^i}(\ind{})(x) \, \left( t^{\frac{\alpha}{2}} + V_i(x) \, e^{-tL_i}(\ind{})(x) \right) \right)^a \, t^{a-1} \, dt \\
		&= \int_0^N \left( V(x) e^{-tL}(\ind{})(x) + t^{\frac{\alpha}{2}} \sum_{i=1}^d e^{-tL^i}(\ind{})(x) \right)^a \, t^{a-1} \, dt \\
		&\leqslant \int_0^N V(x)^a e^{-tL}(\ind{})(x)^a \, t^{a-1} \, dt + \int_0^N t^{\frac{a\alpha}{2}} \left( \sum_{i=1}^d e^{-tL^i}(\ind{})(x) \right)^a \, t^{a-1} \, dt.
	\end{align*}
	To the first integral we apply \eqref{eq:lemma1} and factorization \eqref{eq:fact1}, which lets us estimate the first integral by a constant independent of $x$ and the dimension $d$. To estimate the second integral we fix $x = (x_1, \dots, x_d)$ and divide its coordinates $x_j$ into those whose absolute value is greater than 4 and all others. Say that there are $k$ coordinates greater than 4 and $d-k$ not greater than 4. Then we consider three cases.
	
	First we assume that $k = 0$ and apply \eqref{eq:lemma2} and \eqref{eq:fact1} to get
	\[
		\int_0^N t^{\frac{a\alpha}{2}} \left( \sum_{i=1}^d e^{-tL^i}(\ind{})(x) \right)^a \, t^{a-1} \, dt \leqslant \int_0^N d^a e^{-a(d-1)c_Nt^{\frac{\alpha}{2}+1}} \, t^{\frac{a\alpha}{2}} \, t^{a-1} \, dt \lesssim \frac{d^a}{d^a} = 1.
	\]
	In the last inequality we used
	\begin{equation} \label{eq:aux}
		\int_0^\infty e^{-At^\beta} \, t^\gamma \, dt = \frac{\Gamma \left( \frac{\gamma+1}{\beta} \right)}{\beta A^{\frac{\gamma+1}{\beta}}},
	\end{equation}
	with $A = a(d-1)c_N$, $\beta = \frac{\alpha}{2} + 1$ and $\gamma = \frac{a\alpha}{2} + a - 1$.
	
	Then if $k = d$, we apply \eqref{eq:lemma1} and \eqref{eq:fact1} and use the fact $V_i(x) \geqslant m \cdot 4^\alpha$ which gives
	\begin{align*}
		\int_0^N t^{\frac{a\alpha}{2}} \left( \sum_{i=1}^d e^{-tL^i}(\ind{})(x) \right)^a \, t^{a-1} \, dt &\leqslant \int_0^N d^a e^{-4^\alpha m a c_N t(d-1)} \, t^{\frac{a\alpha}{2}} \, t^{a-1} \, dt \\
		&\lesssim \int_0^N d^a e^{-4^\alpha m a c_N td} \, t^{a-1} \, dt \lesssim 1.
	\end{align*}
	The third case is when $0 < k < d$ in which the estimate is a mixture of the estimates for $k=0$ and $k=d$. Observe that each $(d-1)$-element subsequence of $(x_1, \dots, x_d)$ has at least $k-1$ elements greater than 4 and at least $d-k-1$ elements not greater than 4. By \eqref{eq:lemma1} and \eqref{eq:lemma2} this means that
	\[
		\int_0^N t^{\frac{a\alpha}{2}} \left( \sum_{i=1}^d e^{-tL^i}(\ind{})(x) \right)^a \, t^{a-1} \, dt \leqslant \int_0^N d^a e^{-4^\alpha m a c_N t(k-1)} \, e^{-a c_N (d-k-1)t^{\frac{\alpha}{2}+1}} t^{\frac{a\alpha}{2}} \, t^{a-1} \, dt.
	\]
	Then we use H\"{o}lder's inequality with $p = \frac{d-2}{k-1}$ ($p = \infty$ if $k=1$) and $q = \frac{d-2}{d-k-1}$ ($q = \infty$ if $k=d-1$) to the functions $e^{-4^\alpha m a c_N t(k-1)}$ and $e^{-a c_N (d-k-1)t^{\frac{\alpha}{2}+1}} t^{\frac{a\alpha}{2}}$ with respect to the measure $t^{a-1} \, dt$ which yields
	\begin{align*}
		d^a \int_0^N & e^{-4^\alpha m a c_N t(k-1)} \cdot \, e^{-a c_N (d-k-1)t^{\frac{\alpha}{2}+1}} t^{\frac{a\alpha}{2}} \cdot \, t^{a-1} \, dt \\
		&\lesssim d^a \left(\int_0^N e^{-4^\alpha m a c_N t(d-2)} \, t^{a-1} \, dt \right)^{1/p} \left(\int_0^N e^{-a c_N (d-2)t^{\frac{\alpha}{2}+1}} t^{\frac{a\alpha}{2}} \, t^{a-1} \, dt \right)^{1/q} \\
		&\lesssim d^a \left( \frac{1}{d^a} \right)^{1/p} \left( \frac{1}{d^a} \right)^{1/q} = 1.
	\end{align*}
	Again, in the last inequality we used \eqref{eq:aux}.
\end{proof}

\bibliographystyle{plain}
\bibliography{bib}

\end{document}